\documentclass[11pt,A4paper,reqno]{amsart}
\usepackage{fullpage}
\usepackage{setspace}
\usepackage{datetime}
\usepackage{tikz}
\usetikzlibrary{er}
\usetikzlibrary{arrows.meta}
\usepackage{amsmath}
\usepackage{amssymb}
\usepackage{verbatim,enumitem,graphics}
\usepackage{microtype}
\usepackage{xcolor}
\usepackage[backref=page]{hyperref}
\hypersetup{%
    colorlinks,
    linkcolor={red!50!black},
    citecolor={green!50!black},
    urlcolor={blue!50!black}
}
\usepackage{dsfont}

\usepackage[utf8]{inputenc}

\usepackage[abbrev,msc-links,backrefs]{amsrefs} 
\usepackage{doi}

\renewcommand{\PrintDOI}[1]{\doi{#1}}

\usepackage{fnpct}

\usepackage{MnSymbol}

\setenumerate{leftmargin=*} 

\usepackage{enumitem}

\newtheorem{theorem}{Theorem}[section]
\newtheorem{lemma}[theorem]{Lemma}
\newtheorem{proposition}[theorem]{Proposition}

\newtheorem{conjecture}[theorem]{Conjecture}

\def\({\left(}
\def\){\right)}  

\newcommand{\oldqed}{}
\def\endofClaim{\hfill\scalebox{.6}{$\Box$}}

\newcommand{\cH}{\mathcal{H}}

\newcommand{\cF}{\mathcal{F}}

\newcommand{\cK}{\mathcal{K}}
\newcommand{\cI}{\mathcal{I}}

\newcommand{\cC}{\mathcal{C}}

\newcommand{\PP}{\mathbb{P}}

\newcommand{\ovl}{\overline}

\let\subset\subseteq
\let\epsilon\varepsilon
\let\phi\varphi

\DeclareMathOperator{\cc}{cc}

\usepackage{lineno}
\newcommand*\patchAmsMathEnvironmentForLineno[1]{%
\expandafter\let\csname old#1\expandafter\endcsname\csname #1\endcsname
\expandafter\let\csname oldend#1\expandafter\endcsname\csname end#1\endcsname
\renewenvironment{#1}%
{\linenomath\csname old#1\endcsname}%
{\csname oldend#1\endcsname\endlinenomath}}%
\newcommand*\patchBothAmsMathEnvironmentsForLineno[1]{%
\patchAmsMathEnvironmentForLineno{#1}%
\patchAmsMathEnvironmentForLineno{#1*}}%
\AtBeginDocument{%
\patchBothAmsMathEnvironmentsForLineno{equation}%
\patchBothAmsMathEnvironmentsForLineno{align}%
\patchBothAmsMathEnvironmentsForLineno{flalign}%
\patchBothAmsMathEnvironmentsForLineno{alignat}%
\patchBothAmsMathEnvironmentsForLineno{gather}%
\patchBothAmsMathEnvironmentsForLineno{multline}%
}

\begin{document}
\title{Some results and problems on clique coverings of hypergraphs}

\author{Vojtech R\"{o}dl}

\author{Marcelo Sales}

\thanks{The authors were partially supported by NSF grant DMS 1764385}

\address{Department of Mathematics, Emory University, 
    Atlanta, GA, USA}
\email{\{vrodl|mtsales\}@emory.edu}

\begin{abstract}
For a $k$-uniform hypergraph $F$ we consider the parameter $\Theta(F)$, the minimum size of a clique cover of the edge set of $F$. We derive bounds on $\Theta(F)$ for $F$ belonging to various classes of hypergraphs.
\end{abstract}

\maketitle

\section{Introduction}\label{sec:intro}
Let $G$ be a $k$-uniform hypergraph, a \textit{set representation} of $G$ on a set $T$ is a system of subsets $\{S_v \subset T:\: v \in V(G)\}$ with the property that $$\{v_1, v_2, \ldots, v_k\}\in G \text{ if and only if } \bigcap\limits_{i=1}^k S_{v_i}\neq \emptyset.$$
The \textit{representation number} $\Theta(G)$ of $G$ is the smallest cardinality of a set $T$ which admits a set representation. 

We define the \textit{clique covering number} $\cc(G)$ of a $k$-graph $G$ as the minimum integer $m$ such that there exist complete graphs $C_1,\ldots,C_m$ satisfying $G=\bigcup_{i=1}^m C_i$, i.e., every edge is covered by a clique and no clique contains an edge that is not from $G$.
It is well known that the representation number $\Theta(G)$ is equal to $\cc(G)$ (see e.g. \cites{L, RR} or Section \ref{sec:appendix} for a proof). The parameter $\Theta$ has been subject of
interest for a number of researchers (see e.g. \cites{EGP, A, ER}). In particular it was proved in \cite{A} that $\Theta(G) \leq c_1 d^2 \log n$ for any graph $G$ on $n$ vertices with maximum degree on its complement $\Delta(\ovl{G}) \leq d$. On the other hand, it was proved in \cite{ER} that there are graphs $G$ with $\Delta(\ovl{G})\leq d$ and $\Theta(G) \geq c_2 \frac{d^2}{\log d} \log n$.

Here we will be mainly interested in extending these results to $k$-uniform hypergraphs. Perhaps surprisingly, this turned out to be quite nontrivial, and our results in this direction are far from being optimal for large values of $k$.

Since most of our results regarding $\Theta(G)$ will be formulated in terms of restrictions on the degree of the complement $\ovl{G}$, we will introduce a parameter $\vartheta(G)=\Theta(\ovl{G})$ allowing for simpler formulation of results. Clearly, $\vartheta(G)$ is the minimum cardinality of a set $T$ such that there is a system $\{S_v \subseteq T:\: v \in V(G)\}$ of sets with
\begin{align*}
\{v_1,\ldots , v_k\}\in G\text{ if and only if }\bigcap\limits_{i=1}^k S_{v_i}= \emptyset.
\end{align*}
Alternatively, one can view $\vartheta(G)$ as the minimum cover of the edges of the complement $\ovl{G}$ by independent sets of $G$. In what follows we will often use the latter definition instead of the former. Moreover, we say that $G$ can be \textit{$t$-represented} if there exists a system of independent sets $\{I_j:\:1\leq j \leq t\}$ covering the edges of $\ovl{G}$.

\section{Results}\label{sec:results}
Let $G$ be a $k$-graph. For $S \subset V(G)$, let
\begin{equation*}
    \deg_G(S)=|\{e \in G:\: S \subseteq e\}|
\end{equation*}
be the degree of the set $S$ in $G$ and
\begin{equation*}
   \Delta_i (G) = \max\limits_{|S|=i, S\subseteq V(G)} \deg_G(S)
\end{equation*}
be the maximum degree of an $i$-tuple of vertices of $G$. (Consequently, $\Delta_1(G)=\Delta(G)$ is just the usual maximum degree).

We say that a $k$-graph $G$ is \textit{$d$-balanced} if $\Delta_i (G) \leq d^{\frac{k-i}{k-1}}$ for all $1\leq i \leq k$. In particular, a $d$-balanced graph has maximum degree $d$. Our first result gives almost sharp bounds on $\vartheta(G)$ for the family of $d$-balanced $k$-graphs. Let 
\begin{align*}
    b(n,d,k)=\max\{\vartheta(G):\: \text{$G$ is a $d$-balanced $k$-graph}\}
\end{align*}
be the maximum value of $\vartheta(G)$ over all $d$-balanced graphs with $n$ vertices.

\begin{theorem}\label{th:balanced}
For $n,d,k\geq 2$, there exist positive constants $c_1$ and $c_2$ depending only on $k$ such that
\begin{align*}
\frac{c_1d^{\frac{k}{k-1}}}{\log d} \log \left(\frac{n^{k-1}}{d}\right) \leq b(n,d,k) \leq c_2 d^{\frac{k}{k-1}} \log n
\end{align*}
holds.
\end{theorem}

For the general family of bounded degree hypergraphs, we obtain the following bounds. Let $\vartheta(n,d,k)$ be the maximum value of $\vartheta(G)$, where $G$ is a $k$-graph on $n$ vertices with $\Delta(G)=\Delta_1(G)\leq d$.

\begin{theorem}\label{th:generalcase}
For $n,d,k\geq 3$, there exist positive constants $c_1$ and $c_2$ depending only on $k$ such that
\begin{align*}
    c_1 \frac{d^{\frac{k}{k-1}}}{\log d} \log \left(\frac{n^{k-1}}{d}\right) \leq \vartheta(n,d,k) \leq c_2 d^{\frac{k}{2}} \log n.
\end{align*}
Moreover, if $k$ is even the lower bound can be improved to
\begin{align*}
\frac{c_1 d^2}{\log d} \log \left(\frac{n}{d}\right) \leq \vartheta(n,d,k)
\end{align*}
\end{theorem}

Note that since every $2$-graph $G$ with $\Delta(G)\leq d$ is $d$-balanced, then the bounds of $\vartheta(n,d,k)$ provided in Theorems \ref{th:balanced} and \ref{th:generalcase} are almost sharp for $k=2,3,4$ and $d\ll n$. Finally, we address similar questions for some other classes of hypergraphs. For example, denoting $\vartheta_s(n,k)$ as the maximum value of $\vartheta(G)$, where $G$ runs over all $k$-uniform simple hypergraphs, we show the following:

\begin{theorem}\label{th:linear}
For $k\geq 3$ and $n\geq n_0 (k)$, there exist positive constants $c_1$ and $c_2$ depending only on $k$ such that
\begin{align*}
    \frac{c_1 n^{\frac{k}{k-1}}}{(\log n)^{\frac{1}{k-1}}} \leq \vartheta_s(n,k) \leq c_2 n^{\frac{k}{k-1}}\log n
\end{align*}
holds.
\end{theorem}

The paper is organised as follows. In Section \ref{sec:upper} we discuss the upper bounds of Theorem \ref{th:balanced} and \ref{th:generalcase}, while in Section \ref{sec:lower} we provide their respective lower bounds. Section \ref{sec:steiner} is devoted to the problem of representing special hypergraphs.

\section{Upper Bounds of Theorem \ref{th:balanced} and Theorem \ref{th:generalcase}}\label{sec:upper}
We start with the upper bound of Theorem \ref{th:balanced}. We are going to show that if $G$ is a $d$-balanced $k$-graph, then $\vartheta(G) \leq 2^k k^{k+1}d^{\frac{k}{k-1}} \log n$.

\begin{proposition}\label{prop:upperbalanced}
Let $n,d,k\geq 2$ and $c=2^{k+2}k^{k+1}$. If $G$ is a $d$-balanced $k$-graph on $n$ vertices, then $\vartheta(G)\leq cd^{\frac{k}{k-1}}\log n$.
\end{proposition}

\begin{proof}
We recall that $\vartheta(G)$ is the minimum size of a cover of $\ovl{G}$ by independent sets of $G$. Consider $t= cd^{\frac{k}{k-1}}\log n$ random subsets $W_1,W_2,\ldots, W_t$ of $V:=V(G)$, where each $W_j$ is chosen by selecting vertices from $V$ independently and uniformly with probability $p=\frac{1}{2kd^{1/(k-1)}}$. Our aim is to cover $\ovl{G}$ by modifying the sets $W_j$ for $1\leq j \leq t$. For each $W_j$, construct an independent set as follows: For every edge $e$ in $G[W_j]$ delete all of its vertices from $W_j$. Let $I_j$ be the set obtained after deleting vertices from every edge in $G[W_j]$, i.e., 
\begin{align*}
I_j=W_j\setminus \bigcup_{e\in G[W_j]}e.
\end{align*} 
Clearly, $I_j$ is an independent set. We claim that $\{I_1,\ldots,I_t\}$ is a covering of $\ovl{G}$ with positive probability.

For a $k$-tuple $\ovl{e}=\{x_1,\ldots,x_k\} \in \ovl{G}$, we examine the probability that $\ovl{e} \subseteq I_j$ for some $1\leq j\leq t$. To do that, note that the only reason why a vertex $x_i \in \ovl{e}$ could be deleted is if there exists an edge $e \in G[W_j]$ such that $x_i\in e\cap \ovl{e}$. That is, $\ovl{e}\subseteq I_j$ if $\ovl{e}\subseteq W_j$ and $e\cap \ovl{e}=\emptyset$ for all $e\in G[W_j]$. Hence, we obtain 
\begin{align*}
    \PP(\ovl{e} \subset I_j)\geq p^k \left(1-\sum\limits_{i=1}^{k-1} \Delta_i(G) p^{k-i} \binom{k}{i}\right)\geq p^k\left(1-\sum\limits_{i=1}^{k-1}\frac{\binom{k}{k-i}}{(2k)^{k-i}}\right)>\frac{p^k}{4}.
\end{align*}

Consequently, since all $I_j$ were chosen independently, we obtain that
\begin{align*}
    \PP\left(\bigwedge_{j=1}^t (\ovl{e}\not\subseteq I_j)\right)=\prod_{j=1}^t\PP\left(\ovl{e}\not\subseteq I_j\right)<\left(1-\frac{p^k}{4}\right)^t \leq \exp\left(-\frac{p^kt}{4}\right)=\frac{1}{n^k}. 
\end{align*}

Since there are at most $\binom{n}{k}$ non-edges in $\ovl{G}$ to cover, the probability that one remains uncovered is at most $\binom{n}{k}\frac{1}{n^k}<1$. Consequently, with positive probability all edges of $\ovl{G}$ are covered by $\bigcup_{j=1}^t I_j$.
\end{proof}

The proof of the upper bound of Theorem \ref{th:generalcase} follows similar lines.

\begin{proposition}\label{prop:uppergeneral}
Let $n,d,k\geq 3$, $\delta =\frac{1}{(k-1)2^{k+2}}$ and $c=\frac{2k}{\delta^k}$. If $G$ is a $k$-graph on $n$ vertices with $\Delta(G) \leq d$, then $\vartheta(G) \leq c d^{\frac{k}{2}} \log n$.
\end{proposition}

\begin{proof}

As in the proof of Proposition \ref{prop:upperbalanced}, we want to cover all the edges of $\ovl{G}$ with independent sets of $G$. To this end, we consider $t=cd^{\frac{k}{2}}\log n$ random subsets $W_1,\ldots, W_t$ of $V$, where each $W_j$ is chosen by selecting vertices from $V$ independently and uniformly with probability $p=\frac{\delta}{\sqrt{d}}$. Note that the sets $W_1,\ldots,W_t$ are not necessarily independent. Our aim is to modify those sets $W_j$ to independent sets $I_j$ such that $\bigcup_{j=1}^t I_j$ covers every $k$-tuple in $\ovl{G}$.

To this end, we consider the auxiliary $(k-1)$-graph $H$ given by $V(H)=V$ and 
\begin{align*}
    H=\left\{S \in V^{(k-1)}:\: \deg_{G}(S)\geq \sqrt{d}\right\},
\end{align*}
i.e., $H$ is the $(k-1)$-graph where the edges are the $(k-1)$-tuples of large degree in $G$. The next proposition gives an upper bound on the maximum degree of $H$.

\begin{proposition}\label{prop:auxiliarydeg}
$\Delta(H) \leq (k-1)\sqrt{d}$.
\end{proposition}
\begin{proof}
Assume the contrary and let $v$ be a vertex with $\deg_H (v) > (k-1)\sqrt{d}$. Then by a counting argument we obtain that
\begin{align*}
    \deg_G(v)= \frac{\sum_{v \in S \in V^{(k-1)}}\deg_G(S)}{k-1}\geq \frac{\deg_H(v)\sqrt{d}}{k-1}>d,
\end{align*}
which contradicts the fact that $\Delta(G)\leq d$.
\end{proof}

Now we describe how to construct the independent sets $I_j$. For each $1\leq j\leq t$, we sequentially remove vertices from $W_j$ satisfying the following cleaning strategy:

\begin{center}
    \textbf{\underline{Cleaning Strategy}}
\end{center}

Let $X=W_j$. While there is an edge in $G[X]$ perform one of the two operations:
\begin{description}
    \item[Operation (i)] If there is an edge $g \in G[X]$ containing precisely one $(k-1)$-tuple $S \in H$, then we remove the vertex not in $S$ from $X$, i.e., the vertex in the singleton set $g\setminus S$.
    \item[Operation (ii)] Otherwise, if all the edges in $G[X]$ contain either zero or more than one $(k-1)$-tuple from $H$, then we select an arbitrary edge $g\in G[X]$ and delete an arbitrary vertex from it.
\end{description}

Set $I_j$ to be the resulting $X$ obtained after the process is over. 

\vspace{0.5cm}

Clearly, the set $I_j$ does not contain any edge from $G$ and thus it is independent. It remains to show that with positive probability any edge from the complement $\ovl{G}$ is covered by $\bigcup_{j=1}^tI_j$. 

For a $k$-tuple $\ovl{e}=\{x_1,\ldots,x_k\} \in \ovl{G}$, we want to estimate the probability that $\ovl{e}\subseteq I_j$ for some $1\leq j\leq t$. The following lemma gives us a lower bound on the probability.

\begin{lemma}\label{lem:probIj}
For $\ovl{e} \in \ovl{G}$ and $1\leq j \leq t$,
\begin{align*}
    \PP(\ovl{e}\subseteq I_j)\geq \frac{1}{2}\left(\frac{\delta}{\sqrt{d}}\right)^{k}.
\end{align*}
\end{lemma}

\begin{proof}

In what follows, we will say that the set $X$ \textit{crosses} the pair $(Y,Z)$ if $X$ has non empty intersection with both $Y$ and $Z$. Set $\ovl{e}=\{x_1,\ldots,x_k\}$. We start by defining some auxiliary events. 

For a set $S\subseteq \ovl{e}$ with $S\notin H$ and $|S|\leq k-1$, let 
\begin{align}\label{eq:AS}
    A_S=\left\{\nexists R\subseteq W_j-\ovl{e},\, R\cup S\in G\right\}
\end{align}
be the event that there is no edge $g\in G$ crossing the pair $(\ovl{e}, W_j-\ovl{e})$ with $\ovl{e}\cap g=S$. Similarly, for a set $S\subseteq \ovl{e}$ with $|S|\leq k-2$, let
\begin{align}\label{eq:BS}
    B_S=\left\{\nexists R\subseteq W_j-\ovl{e},\, R\cup S\in H\right\}
\end{align}
be the event stating that there is no $(k-1)$-tuple $h\in H$ crossing the pair $(\ovl{e},W_j-\ovl{e})$ with $\ovl{e}\cap h=S$.

The next claim gives a lower bound on the probability of $\ovl{e}$ be covered by $I_j$.

\begin{proposition}\label{prop:events}
If $\ovl{e} \in \ovl{G}$ is a $k$-tuple satisfying
\begin{enumerate}
    \item[(a)] For every $S\subseteq \ovl{e}$ with $|S|\leq k-1$ and $S\notin H$, the event $A_S$ holds,
    \item[(b)] For every $S\subseteq \ovl{e}$ with $|S|\leq k-2$, the event $B_S$ holds,
    \item[(c)] $\ovl{e}\subseteq W_j$,
\end{enumerate}
then $\ovl{e}\subseteq I_j$. 
\end{proposition}

\begin{proof}Suppose that $\ovl{e}\not\subseteq I_j$. This means that a vertex of $\ovl{e}$ was deleted while performing operations (i) and (ii), i.e., in our cleaning process we deleted a vertex from an edge $g\in G$ with $g\cap\ovl{e}\neq \emptyset$. Since (a) holds, we obtain that $g\cap\ovl{e}$ is an edge of $H$. Moreover, by (b), we have that $g\cap \ovl{e}$ is the only $(k-1)$-tuple of $g$ in $H$. Therefore, we deleted a vertex of $g$ in the operation (i). By the definition of operation (i), we obtain that the deleted vertex was in $g- \ovl{e}$, which contradicts our assumption that the deleted vertex belongs to $\ovl{e}$.
\end{proof}

As a consequence of Proposition \ref{prop:events}, one can estimate the probability of $\ovl{e}\subseteq I_j$ by
\begin{align}\label{eq:probIj}
    \PP(\ovl{e}\subseteq I_j)\geq \PP\left(\{\ovl{e}\subseteq W_j\}\wedge \bigwedge_{S\in \binom{\ovl{e}}{k-1}\setminus H} A_S \wedge \bigwedge_{S\in \binom{\ovl{e}}{k-2}} B_S\right)=p^k\PP\left(\bigwedge_{S\in \binom{\ovl{e}}{k-1}\setminus H} A_S \wedge \bigwedge_{S\in \binom{\ovl{e}}{k-2}} B_S\right).
\end{align}

To compute the probability of the intersection of events $A_S$ and $B_S$, we will estimate the probability of the complementary events $A_S^{c}$ and $B_S^{c}$. Set $\varepsilon=\frac{1}{2^{k+2}}$ and recall that $\delta=\frac{1}{(k-1)2^{k+2}}$. We split the computations into cases depending on the size of $S$:

\noindent\underline{Case 1:} $|S|\leq k-2$.

By using the definitions of $A_S$ in (\ref{eq:AS}) we obtain that
\begin{align}
    \PP(A_S^c)&=\PP\left(\bigvee_{R\cup S\in G} \{R\subseteq W_j-\ovl{e}\}\right) \leq \sum_{R\cup S\in G}\PP(R\subseteq W_j-\ovl{e}) \nonumber\\
    &\leq p^{k-|S|}\deg_G(S) \leq p^2d=\delta^2\leq \epsilon\label{eq:ASc}.
\end{align}
Proposition \ref{prop:auxiliarydeg} and the definition of $B_S$ in (\ref{eq:BS}) give us that
\begin{align}
    \PP(B_S^c)&=\PP\left(\bigvee_{R\cup S\in H} \{R\subseteq W_j-\ovl{e}\}\right) \leq \sum_{R\cup S\in H}\PP(R\subseteq W_j-\ovl{e}) \nonumber\\
    &\leq p^{k-1-|S|}\deg_H(S) \leq p\Delta(H)=(k-1)\delta= \epsilon\label{eq:BSc}.
\end{align}

\noindent\underline{Case 2:} $|S|=k-1$ and $S\notin H$.

By the definition of $A_S$ in (\ref{eq:AS}) and by the fact that $S\notin H$ implies that $\deg_G(S)\leq \sqrt{d}$, we obtain that
\begin{align}
    \PP(A_S^c)&=\PP\left(\bigvee_{R\cup S\in G} \{R\subseteq W_j-\ovl{e}\}\right) \leq \sum_{R\cup S\in G}\PP(R\subseteq W_j-\ovl{e}) \nonumber\\
    &\leq p^{k-|S|}\deg_G(S) \leq p\sqrt{d}=\delta\leq \epsilon\label{eq:ASc2}.
\end{align}

Finally, since there exist at most $2^k$ choices of $S\subseteq \ovl{e}$, by putting (\ref{eq:probIj}), (\ref{eq:ASc}), (\ref{eq:BSc}) and (\ref{eq:ASc2}) together we obtain that
\begin{align*}
    \PP(\ovl{e}\subseteq I_j)\geq p^k\left(1-\sum_{S\in \binom{\ovl{e}}{k-1}\setminus H} \PP(A_S^c)-\sum_{S\in \binom{\ovl{e}}{k-2}}\PP(B_S^c)\right)\geq p^k(1-2^{k+1}\epsilon)\geq \frac{p^k}{2}=\frac{1}{2}\left(\frac{\delta}{\sqrt{d}}\right)^k.
\end{align*}
\end{proof}

Now the remaining part of the proof of Proposition \ref{prop:uppergeneral} is straightforward. By Lemma \ref{lem:probIj}, the probability that $\ovl{e}$ is not covered by $\bigcup_{j=1}^tI_j$ is given by
\begin{align*}
    \PP\left(\bigwedge_{j=1}^t (\ovl{e}\not\subseteq I_j)\right)=\prod_{j=1}^t\PP\left(\ovl{e}\not\subseteq I_j\right)<\left(1-\frac{1}{2}\left(\frac{\delta}{\sqrt{d}}\right)^k\right)^t \leq \exp\left(-\frac{t}{2}\left(\frac{\delta}{\sqrt{d}}\right)^k\right)=\frac{1}{n^k}. 
\end{align*}
Since there are at most $\binom{n}{k}$ non-edges in $\ovl{G}$ to cover, the probability that one remains uncovered is at most $\binom{n}{k}\frac{1}{n^k}<1$. Consequently, with positive probability all edges of $\ovl{G}$ are covered by $\bigcup_{j=1}^t I_j$.
\end{proof}

\section{Lower Bounds of Theorems \ref{th:balanced} and \ref{th:generalcase}}\label{sec:lower}

In this section, we prove the lower bound of Theorems \ref{th:balanced} and \ref{th:generalcase}. The next proposition shows the lower bound of Theorem \ref{th:balanced} and Theorem \ref{th:generalcase} when $k$ is odd.

\begin{proposition}\label{prop:lower}
For $n,d,k\geq 2$, there exists a $d$-balanced $k$-graph $G$ such that
\begin{align*}
\vartheta(G)\geq c\frac{d^{\frac{k}{k-1}}}{\log d}\log \left(\frac{n^{k-1}}{d}\right)
\end{align*}
for a positive constant $c$ depending only on $k$.
\end{proposition}

The proof of Proposition \ref{prop:lower} follows from a counting argument using the next two auxiliary lemmas. Recall that a $t$-representation of a $k$-graph $H$ is a system $\cI=\{I_j\subseteq [n]:\: 1\leq j \leq t\}$ covering the edges of the complement $\ovl{H}$. For integers $n,\alpha, t\geq 1$, define $f(n,\alpha,t)$ as the number of all possible $t$-representations of graphs $H$ with independence number $\alpha(H)\leq \alpha$ (If some graph $H$ does not admit a $t$-representation, then its contribution to the sum is zero).

\begin{lemma}\label{lem:representationnum}
For $\alpha,t \geq 1$ and sufficiently large $n$, $f(n,\alpha,t)\leq \left(2\binom{n}{\alpha}\right)^{t}$.
\end{lemma}

\begin{proof}
Let $\cI=\{I_1,\ldots,I_t\}$ be a set system covering $\ovl{H}$ for some $H$. Since $I_j$ is an independent set of $H$ and $\alpha(H)\leq \alpha$, we obtain that $|I_j|\leq \alpha$. Hence, the number of ways to choose distinct systems $\cI$ is bounded by
\begin{align*}
    f(n,\alpha,t)\leq \left(\sum_{i=1}^\alpha \binom{n}{i}\right)^t\leq \left(2\binom{n}{\alpha}\right)^{t}.
\end{align*}
\end{proof}

Given integers $m,d,k$, we define $\cF_{m,d}^{(k)}$ to be a family of $k$-partite $k$-graphs $F$ on a fixed set of vertices $V=\bigcup_{i=1}^k V_i$ with $|V_i|=m$ for every $1\leq i \leq k$ satisfying:
\begin{enumerate}
    \item[i)] All edges $e \in F$ are transversal to the partition $V_1\cup\ldots \cup V_k$, i.e., $|e\cap V_i|=1$ for all $1\leq i \leq k$.
    \item[ii)] $F$ is linear.
    \item[iii)] $\Delta(F)\leq d$.
\end{enumerate}

Note that since $F$ is linear, the maximum degree satisfies $\Delta(F)\leq d \leq |V_i|=m$. The following lemma gives a lower bound on the size of $\cF_{m,d}^{(k)}$.

\begin{lemma}\label{lem:sizeofFmd}
Let $k,d,m$ be integers with $k\geq 3$ and $m\geq m_0(k)$. Then,
\begin{align*}
    \left|\cF_{m,d}^{(k)}\right|\geq \left(\frac{m^{k-1}}{d}\right)^{\frac{md}{2k^2}}
\end{align*}
holds.
\end{lemma}

\begin{proof}
Set $s=\frac{dm}{2k^2}$. We will construct a graph $F$ from $\cF_{m,d}^{(k)}$ by successively choosing edges $e_1,\ldots, e_s$ transversal to the partition $V_1\cup \ldots \cup V_k$. Suppose that for $\ell <s$ we already constructed $F_{\ell}=\{e_1,\ldots, e_{\ell}\}$ satisfying:

\begin{enumerate}
    \item[i)] $e_i$ is transversal to $V_1\cup \ldots \cup V_k$ for all $1\leq i \leq \ell$.
    \item[ii)] $F_{\ell}$ is linear.
    \item[iii)] $\Delta(F_{\ell})\leq d$.
\end{enumerate}

Now we intend to add an edge $e_{\ell+1}$ to construct the next graph $F_{\ell+1}=F_\ell\cup \{e_{\ell+1}\}$. For each $1\leq i \leq k$, we have
\begin{align*}
    \sum_{x\in V_i}\deg_{F_\ell}(x)=\ell<s= \frac{dm}{2k^2}.
\end{align*}
Consequently, if $X_i=\{x\in V_i:\: \deg_{F_\ell}(x)=d\}$, then we have
\begin{align}\label{eq:sizeofXi}
    |X_i|\leq \frac{m}{2k^2}.
\end{align}

To choose an edge $e_{\ell+1}$ we need to satisfy $|e_i\cap e_{\ell+1}|\leq 1$ for every $1\leq i \leq \ell$. Given $e_i$, there are at most $\binom{k}{2}m^{k-2}$ $k$-tuples $f\in V_1\times\ldots \times V_k$ such that $|e_i\cap f|\geq 2$. Therefore, since $d\leq m$, there are at most 
\begin{align*}
    \ell\binom{k}{2}m^{k-2}<s\binom{k}{2}m^{k-2}\leq \frac{dm^{k-1}}{4}\leq \frac{m^k}{4}
\end{align*}
$k$-tuples of $V_1\times \ldots \times V_k$ that violate condition ii) of $\cF_{m,d}^{(k)}$. Since the addition of any $k$-tuple containing an element of $X_i$ violates condition iii) of $\cF_{m,d}^{(k)}$, we obtain by (\ref{eq:sizeofXi}) that there are at least
\begin{align*}
    \prod_{i=1}^k |V_i\setminus X_i|-\frac{m^k}{4}\geq m^k\left(\left(1-\frac{1}{2k^2}\right)^k-\frac{1}{4}\right)\geq m^k\left(1-\frac{1}{2k}-\frac{1}{4}\right)\geq \frac{m^k}{2}
\end{align*}
valid choices for $e_{\ell+1}$.

Thus there exist at least $\left(m^k/2\right)^s$ sequences of edges $e_1,\ldots, e_s$ forming a graph $F=\{e_1,\ldots, e_s\}$ satisfying conditions i), ii) and iii) of $\cF_{m,d}^{(k)}$. Since the same graph can be obtained by at most $s!$ of these sequences, we have
\begin{align*}
    \left|\cF_{m,d}^{(k)}\right|\geq \frac{\left(\frac{m^k}{2}\right)^s}{s!}\geq \frac{1}{2}\left(\frac{m^{k}}{s}\right)^s\geq \left(\frac{m^{k-1}}{d}\right)^{\frac{md}{2k^2}}.
\end{align*}
\end{proof}

\begin{proof}[Proof of Proposition \ref{prop:lower}]
We will construct a family of $k$-graphs $\cH$ on $n$ vertices of maximum degree $d$ and small independence number. To this end, set $p=\left(d/2\right)^{\frac{1}{k-1}}$ and consider a partition $[n]=V_1\cup\ldots \cup V_k$ with each $|V_i|=n/k$. For each $1\leq i \leq k$, let $H_i$ be the $k$-graph with vertex set $V_i$ consisting of the union of $n/kp$ vertex disjoint cliques $K_{p}^{(k)}$ of size $p$. Let 
\begin{align*}
    \cH=\{F\cup \bigcup_{i=1}^k H_i:\: F\in \cF_{n/k,d/2}^{(k)}\}
\end{align*}
be the family of $k$-graphs obtained by adding a $k$-graph $F\in \cF_{n/k,d/2}^{(k)}$ with $k$-partition $V_1\cup \ldots \cup V_k$ to the union $\bigcup_{i=1}^k H_i$. By Lemma \ref{lem:sizeofFmd}, we have
\begin{align}\label{eq:sizeofcH}
    |\cH|=|\cF|\geq \left(\frac{2}{d}\left(\frac{n}{k}\right)^{k-1}\right)^{\frac{dn}{4k^3}}.
\end{align}

We claim that $H=F\cup \bigcup_{i=1}^k H_i$ is $d$-balanced for every $H\in \cH$. First, note that if $x\in V(H)$, then $x\in V_i$ for some $1\leq i \leq k$ and consequently
\begin{align*}
    \deg_H(x)=\deg_{H_i}(x)+\deg_F(x)\leq \binom{p-1}{k-1}+\frac{d}{2}\leq \frac{d}{2}+\frac{d}{2}=d.
\end{align*}
Thus $\Delta(H)\leq d$.

Now let $S\in [n]^{(\ell)}$ for $2\leq \ell \leq k-1$. Note that $S\subseteq e$ for an edge $e\in \bigcup_{i=1}^k H_i$ if and only if $S$ is a subset of vertices of some $K_{p}^{(k)}$. Consequently, if $S$ is a subset of the vertex set of some $K_{p}^{(k)}$, then $S\subseteq V_i$ for some $1\leq i \leq k$. Thus, by condition i) of $\cF_{n/k,d/2}^{(k)}$ we have that $\deg_F(S)=0$. Hence,
\begin{align*}
    \deg_H(S)\leq \binom{p-\ell}{k-\ell}\leq d^{\frac{k-\ell}{k-1}}.
\end{align*}
Otherwise, if $S$ is not a subset of $K_p^{(k)}$ then
\begin{align*}
    \deg_H(S)=\deg_F(S)\leq 1,
\end{align*}
since $F$ is linear. Therefore, $H$ is a $d$-balanced $k$-graph.

Finally, we turn our attention to the independence number of $H \in \cH$. Note that $\bigcup_{i=1}^k H_i$ is a union of $n/p$ vertex disjoint copies of $K_p^{(k)}$. Thus, since $\bigcup_{i=1}^k H_i\subseteq H$, we have that
\begin{align}\label{eq:indepsizeofH}
    \alpha(H)\leq \alpha\left(\bigcup_{i=1}^k H_i\right) \leq (k-1)\frac{n}{p}
\end{align}
for every $H\in \cH$.

Let $t$ be the minimum integer such that it is possible to $t$-represent any element $H\in \cH$, i.e., $t$ is the minimum integer such that for any $H\in \cH$, there exists a system of independent sets $\{I_j:\: 1\leq j \leq t\}$ with the property that every edge in $\ovl{H}$ is covered by some $I_j$. Lemma \ref{lem:representationnum} applied with (\ref{eq:indepsizeofH}) gives us that there exists at most
\begin{align*}
    \left(2\binom{n}{(k-1)n/p}\right)^t\leq (cp)^{(k-1)nt/p}
\end{align*}
ways to $t$-represent the family $\cH$ for some positive constant $c$ depending on $k$.

Since two graphs $H, H' \in \cH$ can not be represented by the same system $\{I_j:\: 1\leq j \leq t\}$ we obtain that
\begin{align*}
    (cp)^{(k-1)nt/p}\geq |\cH|\geq \left(\frac{2}{d}\left(\frac{n}{k}\right)^{k-1}\right)^{\frac{dn}{4k^3}}.
\end{align*}
Hence, by the fact that $p=(d/2)^{\frac{1}{k-1}}$ we obtain that
\begin{align*}
    t\geq c\frac{d^{\frac{k}{k-1}}}{\log d}\log\left(\frac{n^{k-1}}{d}\right)
\end{align*}
for a positive constant $c$ depending only on $k$. That is, there exists a $d$-balanced $k$-graph $H \in \cH$ such that 
\begin{align*}
    \vartheta(H)\geq c\frac{d^{\frac{k}{k-1}}}{\log d}\log\left(\frac{n^{k-1}}{d}\right)
\end{align*}
\end{proof}

For $k$ even, we can further improve the lower bound for $k$-graphs of bounded maximum degree $d$.

\begin{proposition}\label{prop:evenlb}
For $n, d, k \geq 2$ and $k$ even, there exists $k$-graph $G$ with $\Delta(G)\leq d$ such that
\begin{align*}
    \vartheta(G)\geq c\frac{d^2}{\log d}\log\left(\frac{2n}{kd}\right)
\end{align*}
for a positive constant $c$.
\end{proposition}

\begin{proof}
Suppose that $k=2\ell$ for some integer $\ell$. Proposition \ref{prop:lower} says that there exists a $2$-graph $F$ on $n/\ell$ vertices with $\Delta(F)\leq d$ such that
\begin{align*}
    \vartheta(F)\geq c\frac{d^2}{\log d}\log\left(\frac{n}{\ell d}\right).
\end{align*}

We will construct a $k$-graph $G$ with $\Delta(G)\leq d$ satisfying the inequality of the statement as follows: Let $V(F)=[n/\ell]$. For each $i \in [n/\ell]$, let $V_i=\{v_{i,1},\ldots, v_{i,\ell}\}$ be a set consisting of $\ell$ copies of the vertex $i$. We define $V(G)=\bigcup_{i=1}^{n/\ell} V_i$ and
\begin{align*}
    E(G)=\{V_i\cup V_j:\: \{i,j\} \in F\}.
\end{align*}
That is, the edges of $G$ are the $2\ell$-tuples of the form $V_i\cup V_j$ where $i$ and $j$ are adjacent in $F$. We will prove that $G$ is our desired graph. First note that if $x\in V_i$ for some $i\in [n/\ell]$, then $\deg_G(x)=\deg_F(i)$. Thus $\Delta(G)=\Delta(F)\leq d$.

Now suppose that $G$ can be $t$-represented and let $\cI=\{I_j:1\leq j\leq t\}$ be the system of independent sets covering the edges of $\ovl{G}$. For each $1\leq j \leq t$, we consider the subset $\tilde{I_j}\subseteq V(F)$ given by
\begin{align*}
    \tilde{I_j}=\{i\in [n/\ell]:\: V_i\subseteq I_j\}.
\end{align*}
That is, $\tilde{I_j}$ consists of all vertices $i\in [n/\ell]$ such that $V_i$ is fully contained in $I_j$. Note that $\tilde{I_j}$ is an independent set. Indeed, suppose to the contrary that there exist $i, i' \in \tilde{I_j}$ adjacent in $F$. 
Then, on one hand, by the definition of $\tilde{I_j}$, we have that $V_i\cup V_{i'} \subseteq I_j$. On the other hand, by the definition of $G$, we have that $V_i\cup V_{i'} \in G$, which contradicts the fact that $I_j$ is independent.

We claim that $\tilde{\cI}=\{\tilde{I_j}:\:1\leq j \leq t\}$ is a system of independent sets in $F$ which covers the edges of $\ovl{F}$. In particular, this proves that $\vartheta(G)\geq \vartheta(F)$. Let $\{i,i'\}\in \ovl{F}$ be a non-adjacent pair of vertices in $V(F)$. Then, by the construction of $G$, the $2\ell$-tuple $V_i\cup V_{i'}$ is an edge of $\ovl{G}$. Since $\{I_j:\: 1 \leq j \leq t\}$ covers the edges of $\ovl{G}$, there exists $I_j$ such that $V_i\cup V_{i'}\subseteq I_j$. Thus $\{i,i'\}\subseteq \tilde{I_j}$ and consequently $\{\tilde{I_j}:\: 1\leq j \leq t\}$ covers $\ovl{F}$. Therefore, we obtained a $k$-graph $G$ with $\Delta(G)\leq d$ such that
\begin{align*}
    \vartheta(G)\geq \vartheta(F)\geq c\frac{d^2}{\log d}\log\left(\frac{2n}{kd}\right).
\end{align*}
\end{proof}

\section{Special hypergraphs}\label{sec:steiner}

While for $k$-graphs of bounded degree with $k$ large we still have a significant gap between lower and upper bounds, for Steiner systems one can obtain more precise bounds. Given $1<\ell<k<n$, a \textit{partial Steiner $(n,k,\ell)$-system} is a $k$-graph $S\subseteq [n]^{(k)}$ such that every $\ell$-tuple $P\in [n]^{(\ell)}$ is contained in at most one edge of $S$. An $(n,k,\ell)$-system in which every $\ell$-tuple $P\in [n]^{(\ell)}$ is contained in precisely one edge of $S$ is called a \textit{full Steiner $(n,k,\ell)$-system}. While it is easy to check the existence of partial $(n,k,\ell)$-system, the existence of full $(n,k,\ell)$-system for all admissible parameters $1<\ell<k$ and $n\geq n_0(\ell,k)$ was established only recently in \cite{GKLO} and \cite{K}. In this section we are going to provide bounds for the representation number of partial and full Steiner $(n,k,\ell)$-systems for $\ell=2$.

For $1<\ell<k<n$, we define
\begin{align*}
    s(n,k,\ell)=\max\{\vartheta(S):\: S \text{ is a partial $(n,k,\ell)$-system}\}
\end{align*}
as the maximum value of $\vartheta(S)$, where $S$ runs over all partial $(n,k,\ell)$-systems. Similarly, we define $s^*(n,k,\ell)$ as the maximum value of $\vartheta(S)$, where $S$ runs over all full $(n,k,\ell)$-systems.

\begin{theorem}\label{th:steiner}
Given $k>1$, there exist positive constants $c_1$ and $c_2$ depending only on $k$ such that the following holds:
\begin{enumerate}
    \item[i)] For $n$ sufficiently large,
    \begin{align*}
        c_1\frac{n^{\frac{k}{k-1}}}{(\log n)^{\frac{1}{k-1}}}\leq s(n,k,2) \leq c_2 n^{\frac{k}{k-1}}\log n
    \end{align*}

    \item[ii)] For infinitely many $n$,
    \begin{align*}
        c_1\frac{n^{\frac{k}{k-1}}}{(\log n)^{\frac{1}{k-1}}}\leq s^{*}(n,k,2) \leq c_2 n^{\frac{k}{k-1}}\log n
    \end{align*}
\end{enumerate}
\end{theorem}

Note that Theorem \ref{th:steiner}.i) is just a reformulation of Theorem \ref{th:linear}.
Before we give a proof of the theorem, we are going to introduce two useful results.

Let $K_k^{(\ell)}(m)$ be an $m$-blowup of the complete $\ell$-graph with $k$ vertices, i.e., $K_k^{(\ell)}(m)$ is the graph with vertex set $V=\bigcup_{i=1}^k V_i$, where $|V_i|=m$ for $1\leq i \leq k$, and is such that for every $1\leq i_1<\ldots<i_{\ell}\leq k$ and $x_{i_j}\in V_{i_j}$ the $\ell$-tuple $\{x_{i_1},\ldots,x_{i_\ell}\}$ is an edge. The $K_k^{(\ell)}$-decomposition of $K_k^{(\ell)}(m)$ is a system of pairwise edge disjoint copies of $K_k^{(\ell)}$ covering all edges of $K_k^{(\ell)}(m)$. Equivalently, one can view each such decomposition as a $k$-partite $k$-graph $F$ with vertex partition $V(F)=\bigcup_{i=1}^k V_i$, where $|V_i|=m$ for $1\leq i \leq k$, and such that for every $1\leq i_1<\ldots <i_\ell\leq k$ and vertices $x_{i_j}\in V_{i_j}$ there exists exactly one edge $f \in F$ with $\{x_{i_1},\ldots, x_{i_\ell}\}\subseteq f$.

In \cite{K2}, Keevash obtained strong bounds on the number of such decompositions.

\begin{theorem}[Theorem 2.8, \cite{K2}]\label{th:decomposition}
The number of distinct $K_k^{(\ell)}$-decompositions of $K_k^{(\ell)}(m)$ is given by 
\begin{align*}
\left(\left(e^{1-\binom{k}{\ell}}+o(1)\right)m^{k-\ell}\right)^{m^\ell}
\end{align*}
for sufficiently large $m$.
\end{theorem}

We will also need the following result on existence of $(n,k,\ell)$-systems with bounded independence number.

\begin{theorem}[\cite{GPR}]\label{th:indsteiner}
For $k>\ell\geq 2$, there is a positive constant $c$ depending only on $k$ such that the following holds. If $m=q^2$ for $q$ sufficiently large prime power, then there exists a full Steiner $(m,k,2)$-system $S$ with $\alpha(S)\leq cm^{\frac{k-2}{k-1}}(\log m)^{\frac{1}{k-1}}$.
\end{theorem}

\begin{proof}[Proof of Theorem \ref{th:steiner}]
To establish the upper bound we first note that any full or partial $(n,k,2)$-system $S$ is a $\Delta_1(S)$-balanced $k$-graph. Indeed, observe that for any $2\leq i \leq k-1$ we have $\Delta_i(S)\leq 1 < \Delta_1(S)^{\frac{k-i}{k-1}}$. Also note that for $x\in V(S)$,
\begin{align*}
    \deg(x)\leq \sum_{y\in [n]\setminus\{x\}}\deg(\{x,y\})\leq n-1,
\end{align*}
since $\deg(\{x,y\})\leq 1$ for all $x,y \in V(S)$. Hence $\Delta_1(S)\leq n$ and by the upper bound of Theorem \ref{th:balanced}, there exists positive constant $c_2$ such that
\begin{align*}
    \vartheta(S)\leq c_2\Delta_1(S)^{\frac{k}{k-1}}\log n\leq c_2 n^{\frac{k}{k-1}}\log n.
\end{align*}

Next we prove the lower bounds of $s(n,k,2)$ and $s^*(n,k,2)$. We may assume without changing the asymptotics that $m=n/k=q^2$ for $q$ sufficiently large prime. By Theorem \ref{th:indsteiner} there exists a full $(m,k,2)$-system $S$ with $\alpha(S)\leq cm^{\frac{k-2}{k-1}}(\log m)^{\frac{1}{k-1}}$. Consider $k$ vertex disjoint copies $S_1,\ldots, S_k$ of $S$ and let $V_i=V(S_i)$ for $1\leq i \leq k$. Let $\cF$ be the family of $k$-partite $k$-graphs $F$ with vertex set $V(F)=\bigcup_{i=1}^k V_i$ such that for every $1\leq i<j\leq k$ and $x_i\in V_i$, $x_j\in V_j$, there exists exactly one $f\in F$ such that $\{x_i, x_j\}\subseteq f$. As discussed previously, the family $\cF$ is in one-to-one correspondence with the $K_k^{(2)}$-decompositions of $K_k^{(2)}(m)$. Thus, by Theorem \ref{th:decomposition}, we have
\begin{align}\label{eq:cF}
    |\cF|=\left(\left(e^{1-\binom{k}{2}}+o(1)\right)m^{k-2}\right)^{m^2},
\end{align}
for sufficiently large $m$.

Let $\cH$ be the family of graphs defined by
\begin{align*}
    \cH=\{F\cup \bigcup_{i=1}^k S_i:\: F\in \cF\}.
\end{align*}
We observe that every $H\in \cH$ is a full $(n,k,2)$-system. Let $H=F\cup \bigcup_{i=1}^k S_i$ for some $F\in \cF$ and let $\{x,y\} \subseteq V(H)$. If $\{x,y\} \subseteq V_i$ for some $1\leq i \leq k$, then since $S_i$ is a full Steiner system, there exists exactly one edge $e\in S_i$ such that $\{x,y\}\subseteq e$. Also if $x\in V_i$ and $y\in V_j$ for $i\neq j$, then by the definition of $F$, there exists exactly one edge $f\in F$ such that $\{x,y\}\subseteq f$.

Set $\alpha=k\alpha(S)$ and let $t$ be the minimum integer such that every graph $H \in \cH$ admits a $t$-representation. By Lemma \ref{lem:representationnum} there are $\left(2\binom{n}{\alpha}\right)^t$ ways to $t$-represent graphs in $\cH$. Since every $t$-representation corresponds to a unique graph in $\cH$ and by our assumption that every graph $H\in \cH$ has a $t$-representation, we have by (\ref{eq:cF}) that
\begin{align*}
    2^{t\alpha \log n}\geq \left(2\binom{n}{\alpha}\right)^t\geq |\cH|=|\cF|=\left(\left(e^{1-\binom{k}{2}}+o(1)\right)m^{k-2}\right)^{m^2}\geq 2^{c'n^2\log n}
\end{align*}
Since $\alpha(S)\leq cm^{\frac{k-2}{k-1}}(\log m)^{\frac{1}{k-1}} \leq c''n^{\frac{k-2}{k-1}}(\log n)^{\frac{1}{k-1}}$, we obtain that
\begin{align*}
    t\geq c_1\frac{n^{\frac{k}{k-1}}}{(\log n)^{\frac{1}{k-1}}},
\end{align*}
for a positive constant $c_1$ depending on $k$. Therefore, there exists a full $(n,k,2)$-system $H \in \cH$
such that $\vartheta(H)\geq c_1\frac{n^{\frac{k}{k-1}}}{(\log n)^{\frac{1}{k-1}}}$. Since a full $(n,k,2)$-system is a partial $(n,k,2)$-system, we obtain that
\begin{align*}
    s(n,k,2)\geq s^*(n,k,2)\geq c_1\frac{n^{\frac{k}{k-1}}}{(\log n)^{\frac{1}{k-1}}}.
\end{align*}

\end{proof}

\section{Concluding Remarks}

\subsection{Representation of sparse hypergraphs of high uniformity}

In Sections \ref{sec:upper} and \ref{sec:lower} we determined upper and lower bounds for the family of $k$-graphs on $n$ vertices of bounded degree $d$. In particular, if we define
\begin{align*}
    f_k(d)=\lim_{n\rightarrow \infty}\frac{\vartheta(n,d,k)}{\log n},
\end{align*}
then we established almost sharp bounds for $k=2,3,4$, i.e., $\Omega\left(d^2/\log d\right)\leq f_2(d)\leq O(d^2)$, $\Omega\left(d^{3/2}/\log d\right) \leq f_3(d) \leq O(d^{3/2})$ and $\Omega\left(d^2/\log d\right)\leq f_4(d)\leq O(d^2)$.

Unfortunately, for $k\geq 5$, the bounds we have at the moment are worse. More precisely,
\begin{align*}
    \Omega\left(\frac{d^2}{\log d}\right)\leq f_k(d)\leq O\left(d^{k/2}\right)
\end{align*}
for $k$ even and 
\begin{align*}
    \Omega\left(\frac{d^{\frac{k}{k-1}}}{\log d}\right)\leq f_k(d)\leq O\left(d^{k/2}\right)
\end{align*}
for $k$ odd. It would be interesting to close the gap between the bounds for $f_k(d)$.

\subsection{Representation of $k$-partite $k$-graphs}

Let $\cK(n,d,k)$ be the family of $k$-partite $k$-graphs $G$ on $n$ vertices with bounded maximum degree $\Delta(G)\leq d$. As in the previous case, it would be interesting to find good bounds on $\vartheta(G)$ for $G\in \cK(n,d,k)$. Let 
\begin{align*}
    g_k(d)=\lim_{n\rightarrow \infty} \frac{\max\{\vartheta(G):\: G\in \cK(n,d,k)\}}{\log n}.
\end{align*}
A similar counting argument as the one used in Section \ref{sec:lower} yields that $g_k(d)=\Omega(d)$. For $k=2$, similar techniques as in Section \ref{sec:upper} give the upper bound which is linear in $d$, i.e., $g_2(d)=O(d)$. This leads us to believe that perhaps the same could be true for $k\geq 3$.

\begin{conjecture}
$g_k(d)=O(d)$ for $k\geq 3$.
\end{conjecture}

\bibliography{sources}

\section{Appendix}\label{sec:appendix}

\begin{proposition}
Let $F$ be a $k$-graph for $k\geq 2$. Then $\Theta(F)=\cc(F)$.
\end{proposition}

\begin{proof}
Let $\cC=\{C_1,\ldots,C_t\}$ be a system of cliques covering all edges of $F$. To each vertex $v \in V(F)$ assign the set $S_v$ of all cliques containing $v$. Since $\cC$ is a covering of the edges, if $\{v_1,\ldots,v_k\} \in F$, then there exists a clique $C_i$ containing $\{v_1,\ldots,v_k\}$. Thus $C_i \in \bigcap_{i=1}^k S_{v_i}$ and hence $\bigcap_{i=1}^k S_{v_i}\neq \emptyset$. On the other hand, by construction, $S_{v_j}=\{C_i\in\cC:\: v_j \in C_i\}$ and therefore the set $\bigcap_{i=1}^k S_{v_i}$ contains all cliques containing $\{v_1,\ldots, v_k\}$. Consequently, if $\bigcap\limits_{j=1}^{k}S_{v_j}=\emptyset$, then there is no clique containing $\{v_1,\ldots,v_k\}$, which implies that $\{v_1,\ldots,v_k\}\notin F$. Hence, we proved that $\cC$ is a set representation of $F$ and we obtain that \begin{align*}
\Theta(F) \leq \text{cc}(F).
\end{align*}

To prove the opposite inequality, let $V(F)=\{v_1,\ldots,v_n\}$ and let $t$ be the an integer such that $F$ admits a set representation $\{S_{v_i} \subset [t]:\: 1 \leq i \leq n\}$. For each $s\in [t]$, we define the set $C(s)=\{v_i:\: s\in S_{v_i}\}$. Since $s$ belongs to the intersection of any $k$ sets of the form $S_{v}$ for $v\in C(s)$, we obtain that $F[C(s)]$ is a clique. We claim that $\cF=\left\{F[C(1)],\ldots,F[C(s)]\right\}$ is a clique covering of $F$. Indeed, if $\{v_{i_1},\ldots,v_{i_k}\}\in F$ is an edge of $F$, then $\bigcap_{j=1}^k S_{v_{i_j}}\neq \emptyset$. Let $s\in\bigcap_{j=1}^k S_{v_{i_j}}$. Then clearly $\{v_{i_1},\ldots,v_{i_k}\}\in F[C(s)]$. Now if $\{v_{i_1},\ldots,v_{i_k}\}\notin F$, then $\bigcap_{j=1}^k S_{v_{i_j}}= \emptyset$ and consequently there is no $s$ such that $\{v_{i_1},\ldots,v_{i_k}\}\subseteq C(s)$. Hence, $\cF$ is a clique covering and $\cc(F)\leq \Theta (F)$ follows.
\end{proof}

\end{document}